\documentclass[oneside,11pt]{article} 
\usepackage{makeidx}
\makeindex
\usepackage[utf8]{inputenc} 
\usepackage{amsthm}
\usepackage{amsmath}
\numberwithin{equation}{section}
\usepackage[colorlinks=true, linkcolor=blue, citecolor=blue, urlcolor=blue]{hyperref}
\usepackage[backend=biber, style=numeric, sorting=nyt, sortcites=true, maxbibnames=999, minbibnames=999]{biblatex}
\usepackage{url}
\usepackage{xcolor}
\usepackage{geometry} 
\usepackage{graphicx} 
\usepackage{booktabs} 
\usepackage{array} 
\usepackage{paralist} 
\usepackage{verbatim}
\usepackage{subfig} 
\usepackage{authblk}
\usepackage{amsmath}
\usepackage{amsfonts}
\usepackage{amssymb}
\usepackage{amsthm}
\usepackage{mathrsfs} 
\usepackage{fancyhdr} 
\numberwithin{equation}{section} 
\usepackage{sectsty}
\allsectionsfont{\sffamily\mdseries\upshape} 
\usepackage[nottoc,notlof,notlot]{tocbibind} 
\usepackage[titles,subfigure]{tocloft}

\theoremstyle{remark}
\newtheorem{remark}{Remark}
\theoremstyle{theorem}
\newtheorem{theorem}{Theorem}
\theoremstyle{proposition}
\newtheorem{prop}{Proposition}
\theoremstyle{lemma}
\newtheorem{lemma}{Lemma}
\theoremstyle{defin}
\newtheorem{defin}{Definition}
\theoremstyle{hyp}

\theoremstyle{coro}

\theoremstyle{example}

\title{\bf A Caveat on Metrizing Convergence \\ in Distribution on Hilbert Spaces}

\newcommand{\R}{\mathbb{R}}

\usepackage{etoolbox}
\author[1]{Federico Bassetti}
\author[2]{Solesne Bourguin}
\author[3]{Simon Campese}
\author[4]{Giovanni Peccati}
\affil[1]{Department of Mathematics, Politecnico of Milan}
\affil[2]{Department of Mathematics and Statistics, Boston University}
\affil[3]{Institute of Mathematics, Hamburg University of Technology}
\affil[4]{Department of Mathematics, University of Luxembourg}
\date{}
\begin{document}
\maketitle
\begin{abstract}
  {We consider Sobolev-type distances on probability measures over separable Hilbert spaces involving the Schatten-$p$ norms, which include as special cases a distance first introduced by Bourguin and Campese (2020) when $p=2$, and a distance introduced by Gin\'e and Leon (1980) when $p=\infty$.
  Our analysis shows that, unless $p=\infty$, these distances fail to metrize convergence in distribution in infinite dimensions. This clarifies several inconsistencies and misconceptions in the recent literature that arose from confusion between different types of distances.
}
\end{abstract}

\section{Introduction}\label{s:intro} 

In \cite{bourguincampese}, the second and third authors of the present paper introduced a probabilistic distance, noted $d_2(X,Y)$ and formally identified in Definition \ref{d:distances} and Remark \ref{r:clar} below, between the distributions of two random elements $X,Y$ taking values in a real separable Hilbert space $K$. As we will see, such a distance is obtained by considering the supremum of the quantity $|\mathbb{E}[f(X)] - \mathbb{E}[f(Y)]|$, where $f$ runs over a certain class of twice Fr\'echet differentiable test functions $f$, such that the second derivative $D^2 f$ induces a Hilbert-Schmidt operator on $K$ (see Definition \ref{d:HSD} for a precise formulation). The distance $d_2$ is claimed in \cite[Lemma 3.1]{bourguincampese} to metrize convergence in distribution on $K$.

\smallskip 

The use of the distance $d_2$ allowed the authors of \cite{bourguincampese} to develop an infinite-dimensional version of the so-called {\it Malliavin-Stein method} (see e.g. \cite{npbook}), including Fourth Moment and Breuer-Major type quantitative CLTs. Further uses of the distance $d_2$, and of the techniques initiated in \cite{bourguincampese}, have appeared in subsequent years (and are continuing to appear at the time of writing), spanning fields as diverse as stochastic analysis on configuration spaces, the geometry of random fields, and the large-scale study of randomly initialized neural networks --- see \cite{bourguincampesethanh, BourguinMarinucciDurastanti, cammarotaMarinucciSalviVigogna, caponera, düker2024breuermajortheoremshilbertspacevalued, FHMNP, vidotto2025functionalsecondordergaussianpoincare}\footnote{We stress that $d_2$ is only needed for \cite[formula~(3.27)]{FHMNP}; all other results in \cite{FHMNP} are independent of both $d_2$ and formula~(3.27) therein.
}.

\smallskip

The purpose of this paper is to provide a formal counterexample to \cite[Lemma 3.1]{bourguincampese}, by showing that if ${\rm dim}\, K = +\infty$, the distance $d_2$ {\it cannot be employed} to deduce convergence in distribution on $K$. \footnote{The proof of \cite[Lemma 3.1]{bourguincampese} is based on an erroneous interpretation of a result in \cite{CoutinDecreusefond}, where the considered distance is not the distance $d_2$, but rather the distance $\rho_\infty$ defined below.}

The distance $d_2$ is closely related to the one introduced in \cite{gineleon}, denoted $\rho_\infty$ in Definition~\ref{d:distances} below. Both may be viewed as variants of Zolotarev-type distances of order~$2$, see  \cite{Zolotarev76,drmota}. Conceptually, they share the same structure: in each case, the test functions $f$ are characterized by the fact that  $D^2 f$, regarded as a bounded bilinear operator on $K$, has norm at most~$1$. 
The essential distinction lies in the choice of the norm, a seemingly minor difference that has crucial consequences: $\rho_\infty$ metrizes weak convergence, whereas $d_2$ does not if ${\rm dim}\, K = +\infty$.

Our main result shows, in greater generality, that the integral probability distances $\rho_{p}$, defined analogously to $d_{2}$ with twice Fréchet differentiable test functions whose second derivative belongs to the Schatten $p$-class ($1 \leq p < \infty$), are not strong enough to metrize convergence in distribution. We state this result below and postpone both the precise definition of $\rho_p$ and its proof to the following sections.

\begin{theorem}\label{t:t} Suppose ${\rm dim}\, K = +\infty$ and let $p \in [1,\infty)$. Then, $\rho_p$ does not metrize convergence in distribution on $K$.
\end{theorem}

As a consequence of Theorem \ref{t:t} (note that $\rho_{2} = d_2$, see Remark \ref{r:clar}), the proofs of the functional CLTs deduced in the references \cite{bourguincampesethanh, BourguinMarinucciDurastanti, cammarotaMarinucciSalviVigogna, caponera, düker2024breuermajortheoremshilbertspacevalued, FHMNP} through the use of the distance $d_2$ are not valid. The proof of Theorem \ref{t:t} (as given in Section \ref{ss:proof} below) is by contradiction, and exploits the fact that the metrization property of $\rho_p$, for any $p\in [1, \infty)$, would negate some classical criteria for the convergence in distribution of Hilbert space-valued Gaussian random elements, see \cite[Chapter 3]{BogachevGaussian}.

\smallskip

Before giving a precise definition of the probabilistic metrics we work with, we recall some important facts on linear and bilinear forms and on derivatives of $K$-valued functions.
In particular, 
Lemma~\ref{l:test} below identifies the exact technical obstruction that prevents the classical metrization arguments used in \cite{gineleon, CoutinDecreusefond, drmota} from carrying over to the case of $\rho_p$ for $1\leq p <+\infty$.

\smallskip

\section{Preliminaries}\label{s:preliminaries}

\subsection{Linear mappings and bilinear forms}\label{ss:linear}

In this section, we recall some basic notions about linear mappings and bilinear forms on Banach and Hilbert spaces. Our main reference is \cite[Chapters IV and XIII]{lang}; see also \cite{dieudonne, doria, zeidler, Mishura}

\medskip

Given two real separable Banach spaces $E,F$, we write $L(E,F)$ to indicate the class of all linear continuous (or, equivalently, bounded) mappings from $E$ to $F$. As usual, we endow $L(E,F)$ with the operator norm 
\begin{equation}\label{e:opnorm}
\|T\|_{op} := \sup_{x\, : \, \|x\|_E = 1} \|Tx\|_F.
\end{equation}
It is a well-known fact that, since $F$ is Banach, then $(L(E,F), \|\cdot \|_{op} )$ is also a Banach space. Elements of $L(E,F)$ are also commonly called {\it bounded linear operators}, and we will use the two terms interchangeably.

\medskip 

 Given three real separable Banach spaces $E, F, G$, we write $L(E,F; G)$ to indicate the class of bilinear continuous (or, equivalently, bounded) mappings from the Cartesian product $E\times F$ (endowed with the product Banach topology) to $G$. We endow $L(E,F; G)$ with the norm
\begin{equation}\label{e:binorm}
\|\varphi\| := \sup_{x,y\, : \, \|x\|_E = \|y\|_F = 1} \|\varphi(x,y)\|_G, 
\end{equation}
so that $(L(E,F; G), \|\cdot\|)$ is also a Banach space (since $G$ is Banach). It is a classical result (see e.g. \cite[p. 67]{lang}) that $L(E,F; G)\simeq L(E, L(F, G))$ where `$\simeq$' indicates the existence of a Banach space isomorphism (that is, a norm-preserving, continuous linear mapping whose inverse is also continuous and linear). When $E=F$, we will say that $\varphi\in L(E,E; G)$ is {\it symmetric} if $\varphi(x,y) = \varphi(y,x)$, for all $x,y\in E$.

\medskip 

For the rest of the paper, we will denote by $K$ a real, separable Hilbert space, with inner product $\langle \cdot, \cdot\cdot \rangle_K$ and norm $\|\cdot \|_K := \langle \cdot, \cdot \rangle^{1/2}_K $. We recall that for $p \in [1,\infty)$ an operator $A\in L(K,K)$ is in the $p$-th \textit{Schatten class} $\mathcal{S}_p(K)$ if
\begin{equation}\label{e:pnorm}
   \left\lVert A \right\rVert_p = \operatorname{Tr}(\left| A \right|^p)^{1/p} = \left( \sum_{n=1}^{\infty} s_n(A)^p  \right)^{1/p} < \infty,
\end{equation}
where $\left| A \right| = \sqrt{A^{\ast}A}$ is the absolute value of $A$, $\operatorname{Tr}$ denotes the trace on $K$ and $(s_n(A))_{n \in \mathbb{N}}$ is the sequence of singular values of $A$, i.e. the eigenvalues of $\left| A \right|$ in decreasing order. It is well known that all Schatten class operators $A$ are compact. As usual, we define $\left\lVert A \right\rVert_{\infty} = \left\lVert A \right\rVert_{op}$ and $\mathcal{S}_{\infty}(K) = L(K,K)$. As $\left\lVert A \right\rVert_p \leq \left\lVert A \right\rVert_{p'}$ whenever $1 \leq p' \leq p \leq \infty$, the Schatten classes are increasing in $p$. Notable special cases include the trace class operators ($p=1$) and the {\it Hilbert-Schmidt operators} ($p=2$), where in the latter case the norm can also be written as $\left\lVert A \right\rVert_2 = \sum_{i=1}^\infty \|A e_i\|^2_K$ for some (and, hence, all) orthonormal basis $\{e_i\}$ of $K$. We observe that the class of all Hilbert-Schmidt operators on a real separable Hilbert space is itself a real separable Hilbert space, with inner product given by
\begin{equation}\label{e:innerHS}
\langle A,B\rangle_{HS} := \sum_{i=1}^\infty \langle Ae_i, Be_i\rangle_K,
\end{equation}
where $\{e_i\}$ is an arbitrary orthonormal basis of $K$ (the definition of $\langle \cdot, \cdot\cdot\rangle_{HS}$ is independent of the choice of the basis); we will henceforth use the symbol $\mathcal{HS}(K)$ to denote such a Hilbert space. We will also write $K\otimes K$ and $K\odot K$, respectively, to indicate the {\it tensor product} and the {\it symmetric tensor product} of $K$ with itself (see e.g. \cite[Appendix E]{Janson} for a concise and probability-oriented discussion of tensor products).

\medskip

The following elementary result (whose proof is left to the reader) plays an important role in the sequel. Its statement is provided in order to maintain the paper as self-contained as possible.

\begin{lemma}[Representation of Bilinear Forms]\label{l:bilinear} Let the above notation and assumptions prevail.
\begin{enumerate}
    \item[\rm (1)] For all $\varphi \in L(K,K; \R)$ (that is, for all real-valued bilinear forms $\varphi$ on the Hilbert space $K$), there exists a unique $A_\varphi \in L(K,K)$ such that 
    \begin{equation}\label{e:aphi}
    \varphi(x,y) = \langle A_\varphi x, y \rangle_K, \quad x,y\in K, 
    \end{equation}
    and one has moreover that $\| A_\varphi \|_{op} = \|\varphi\| $, where we have used the notation \eqref{e:opnorm} and \eqref{e:binorm}.
    \item[\rm (2)] Let $\varphi, A_{\varphi}$ be as at Point {\rm (1)}. Then, $\varphi$ is symmetric if and only if $A_\varphi$ is self-adjoint.
    \item[\rm (3)] Let Let $\varphi, A_{\varphi}$ be as at Point {\rm (1)}. Then, $A_\varphi$ is Hilbert-Schmidt if and only if there exists $ H_\varphi \in K\otimes K$ such that
    \begin{equation}\label{e:biliHS}
    \varphi(x,y) = \langle H_\varphi, x\otimes y\rangle_{K\otimes K},
    \end{equation}
    and, in this case, one has that $\varphi$ is symmetric if and only if $H_\varphi \in K\odot K$. 
\end{enumerate}
\end{lemma}

It is easily seen that, if it exists, the element $H_\varphi$ appearing in \eqref{e:biliHS} is unique; moreover, one has that (using the notation \eqref{e:opnorm}--\eqref{e:binorm}) 
\begin{equation}\label{e:chain}
\|\varphi\| \leq \|A_\varphi\|_{op}\leq \|A_\varphi\|_{HS} = \|H_\varphi \|_{K\otimes K},
\end{equation}
where $\|\cdot\|_{HS}$ the norm associated with the inner product defined in \eqref{e:innerHS}. In general, the inequality $\|\varphi\| \leq  \|H_\varphi \|_{K\otimes K}$ cannot be reversed.

\subsection{Differentiable mappings}\label{ss:diff}

Let $E$ be a real separable Banach space. A mapping $f : E\to \R$ is {\it differentiable} (in the Fr\'echet sense) at $x\in K$, if there exist $Df(x) \in L(E, \R) $ such that 
$$
\lim_{\|h\|_E\to 0} \frac{f(x+h) - f(x) - Df(x)(h)}{\|h\|_E} = 0.
$$
We write $f\in C^1(E, \R)$ is $f$ is differentiable at every point $x\in E$ and the mapping $E\to L(E, \R) : x\mapsto Df(x)$ is continuous. We observe that, if $E = K$ is a Hilbert space, then $Df(x)(h) = \langle H_x , h\rangle_K$ for some unique $H_x\in K$ (by the Riesz Representation Theorem), in such a way that $\|Df(x)\|_{op} = \|H_x\|_K$ (see \eqref{e:opnorm}). 

\medskip 

A mapping $f \in  C^1(E, \R)$ is {\it twice differentiable} at $x\in K$, if there exists a symmetric $\varphi_x \in L(E,E; \R)\simeq L(E; L(E, \R))$ such that
$$
\lim_{\|h\|_E\to 0} 
\sup_{\|y\|_E = 1} \left| \frac{Df(x+h)(y) - Df(x)(y) - \varphi_x(h,y)}{\|h\|_E} \right| = 0,
$$
and in this case we write $D^2f(x) := \varphi_x$ to indicate the second derivative of $f$ at $x$. Given $f: E\to \R$, we write $f\in C^2(E, \R)$ if $f\in C^1(E, \R)$, $f$ is twice differentiable at every $x\in E$, and the mapping $E\to L(E,E; \R) : x\mapsto D^2f(x)$ is continuous. 

\smallskip

\begin{defin}\label{d:HSD}{\rm Assume that $E = K$ is a Hilbert space. Then, we say that a mapping $f\in C^2(K, \R)$ {\it admits a $p$-th Schatten class second derivative} if, for all $x\in K$, one has that the operator $A_{\varphi_x}$ associated to the bilinear form $\varphi_x = D^2(f(x))$ via \eqref{e:aphi} is in $\mathcal{S}_p(K)$. In this case, we write (by a slight abuse of notation)}
\begin{equation}\label{e:abuse}
\|D^2f(x)\|_{p} := \| A_{\varphi_x}\|_{p}.
\end{equation}
\end{defin}

When ${\rm dim}\, K=+\infty$, an elementary example of a mapping $f\in C^2(K, \R)$ that {\it does not admit} a second derivative in a $p$-th Schatten class for some $p \in [1,\infty)$ is given by
$$
x\mapsto  \|x\|^2_K,\quad x\in K,
$$
for which one has that 
\begin{equation}\label{e:conti}
D\,\|x\|^2_K(h) = 2\langle x,h\rangle_K, \quad \mbox{and}\quad D^2\,\|x\|^2_K(h,k) = \varphi(h,k) = 2\langle h,k\rangle_K,
\end{equation}
in such a way that, in the notation of \eqref{e:aphi}, $A_\varphi = 2 I$, where $I$ is the identity operator on $K$ (note that, trivially, $I$ is compact if and only if $K$ is finite-dimensional). The following result is proved in Section \ref{s:proofs} by leveraging a variation of the second relation in \eqref{e:conti}.

\begin{lemma}[Smooth Radial Mappings]\label{l:test} Let $K$ be an infinite-dimensional real separable Hilbert space, and let $\psi : \R_+ \mapsto \R : r\mapsto \psi(r)$ be of class $C_b^\infty$ (that is, $\psi$ is bounded and infinitely differentiable, with bounded derivatives of every order). For a fixed $y\in K$, define the mapping
\begin{equation}\label{e:radialsmooth}
f : K\to \R : x\mapsto f(x) := \psi\left(\|x -y\|^2\right).
\end{equation}
Then, $f\in C^2(K,\R)$. Moreover, for each $x\in K$ the second Fr\'echet derivative
$D^2 f(x)$ (a bounded bilinear form on $K\times K$) can be identified via the
Riesz isomorphism with a bounded self-adjoint operator on $K$.
In this sense, $D^2 f(x)$ is a compact operator on $K$ for every $x\in K$
if and only if $\psi$ is constant.

\end{lemma}

\section{Proof of the main result}\label{s:main}
From now on, every random element is assumed to be defined on a common probability space $(\Omega, \mathcal{F}, \mathbb{P}) $, with $\mathbb{E}$ denoting the expectation with respect to $\mathbb{P}$.

\subsection{Probabilistic distances on Hilbert spaces}\label{ss:proba}

Let $K$ be a real separable Hilbert space, that we endow with the Borel $\sigma$-field associated with the norm $\|\cdot \|_K$. As usual, given a collection $\{X_n,X : n\geq 1\}$ of $K$-valued random elements, we say that $X_n$ {\it converges in distribution} to $X$ if $\mathbb{E}[g(X_n)] \to \mathbb{E}[g(X)] $, for all $g : K\to \R$ continuous and bounded. We write $\mathcal{P}(K)$ to indicate the class of all probability measures on $K$. Given a distance $d : \mathcal{P}(K)\times \mathcal{P}(K)\to \R_+: (P,Q)\mapsto d(P,Q)$ on $\mathcal{P}(K)$, we say that $d$ {\it metrizes convergence in distribution} if, for every collection $\{X_n,X : n\geq 1\}$ of $K$-valued random elements, one has that $X_n$ converges in distribution to $X$ if and only if $d(X_n,X)\to 0$, where we have used the convention
\begin{equation}\label{e:convention}
d(X,Y) :=  d\left({\rm Law}(X), {\rm Law}(Y)\right).
\end{equation}

\smallskip

We will now define the announced distances $\rho_{p}$ on $\mathcal{P}(K)$.

\begin{defin}\label{d:distances}{\rm Let the above notation and assumptions prevail. For $p  \in [1,\infty]$, the distance $\rho_{p} : \mathcal{P}(K)\times \mathcal{P}(K)\to \R_+$ is defined as follows: given $P,Q\in \mathcal{P}(K)$ and $X,Y$ such that ${\rm Law}(X) = P$ and ${\rm Law}(Y) = Q$,
$$
\rho_{p}(P,Q) = \rho_{p}(X,Y) = \sup_{f\in \mathcal{F}_{p} } \left| \mathbb{E}[f(X)] - \mathbb{E}[f(Y)] \right|,
$$
where $\mathcal{F}_{p}$ denotes the class of all $f\in C^2(K, \R)$ that admit a second derivative in the $p$-th Schatten class (see Definition \ref{d:HSD}) and such that
$$
\sup_{x\in K} \left(  \|Df(x)\|_{op} +  \|D^2f(x)\|_{p}  \right) \leq 1,
$$
where we have adopted the notation \eqref{e:abuse}.
}
\end{defin}

\begin{remark}\label{r:clar} The symbol $d_2$ in \cite{bourguincampese} indicates the distance denoted by $\rho_2$ in Definition \ref{d:distances}.
\end{remark}

Note that, as the Schatten classes are increasing in $p$, one has that $\rho_{p} \leq \rho_{p'}$ whenever $p \leq p'$.

The distance $\rho_{\infty}$ is an element (up to some minor variations) of the class of discrepancies considered in \cite{CoutinDecreusefond, gineleon} (see also \cite{drmota}). It is a well-known fact (see e.g. \cite[Theorem 2.4]{gineleon} or \cite[Theorem 7]{CoutinDecreusefond}) that $\rho_{\infty}$ metrizes convergence in distribution on every separable real Hilbert space $K$. As previously discussed, Theorem~\ref{t:t} --- whose proof is the subject of the next section --- establishes that for $p \in [1,\infty)$, the distance $\rho_{p}$ metrizes convergence in distribution {only} in the trivial finite-dimensional case, that is, when $K$ has finite dimension and the choice of norm is thus irrelevant.

\begin{remark}{\rm A crucial step in the proof of the metrization property of $\rho_{\infty}$ in
\cite{gineleon} (see, specifically, Theorem~2.4 therein; the symbol $d_{2}$ there corresponds to our notation $\rho_{\infty}$) is the uniform approximation
of indicators of finite intersections of open balls in $K$ by finite products of smooth radial test functions of the form~\eqref{e:radialsmooth}, belonging to the class $\mathcal{F}_{\infty}$. Lemma~\ref{l:test} pinpoints the precise reason why
this approximation strategy {breaks down} when one works with the distances $\rho_{p}$ with $p \in [1,\infty)$ instead of $\rho_{\infty}$. Heuristically, this is due to the fact that the Schatten classes are dense in the space of compact operators, which in turn form a closed and hence non-dense subspace of the bounded operators. Consequently, test functions in $\mathcal{F}_{p}$ for $p \in [1,\infty)$ cannot approximate arbitrary test functions from the metrizing class $\mathcal{F}_{\infty}$ in the corresponding Sobolev-type distance $\left\lVert f \right\rVert = \sup_{x \in K} \left(  \left\lVert Df(x) \right\rVert_{\operatorname{op}} + \left\lVert D^{2}f(x) \right\rVert_{\infty} \right)$. }
\end{remark}

\subsection{Proof of Theorem \ref{t:t}}\label{ss:proof}

We recall that the distribution of a centered Gaussian random element $X$, with values in the Hilbert space $K$, is completely determined by its {\it covariance operator} $S=S_X$, which is defined as the unique symmetric, positive and trace-class (and therefore Hilbert-Schmidt) linear operator $S : K\to K$ such that, for all $g\in K$, $\langle X, g\rangle_K$ is a centered Gaussian random variable with variance $\langle S g, g\rangle_K\geq 0$ (see e.g. \cite[Chapter 1]{daprato}). To show Theorem \ref{t:t}, we will need the following classical characterization of convergence in distribution for centered Gaussian elements (see \cite[Example 3.8.13 and Example 3.8.15]{BogachevGaussian}, as well as \cite[Lemma 4.4]{DNPR}). The statement makes use of the notation introduced in \eqref{e:opnorm} and \eqref{e:innerHS}.

\begin{prop}\label{p:boga} Let $\{X,X_n : n\geq 1\}$ be a sequence of centered Gaussian elements with values in $K$ and let $\{S, S_n : n\geq 1\}$ be the collection of their covariance operators. Then, the following three properties are equivalent, as $n\to \infty$:
\begin{itemize}
\item[\rm (i)] $X_n$ converges in distribution to $X$;

\item[\rm (ii)] $\| \sqrt{S_n} - \sqrt{S}\|_{HS}\to 0$;

\item[\rm (iii)] $\| {S_n} - {S}\|_{op}\to 0$
and $\mathbb{E}[\|X_n\|^2_{K}]\to \mathbb{E}[\|X\|^2_{K}].$
\end{itemize}
\end{prop}

We will also need the following estimate, which is a generalization of \cite[Proposition 5.13]{FHMNP}.

\begin{prop}\label{p:ibp} Let $X_1, X_2$ be two centered Gaussian elements with values in $K$, and let $S_1,S_2$ be their covariance kernels. Then, for any $p \in [1,\infty]$, one has
$$
{\rho}_p(X_1,X_2) \leq \frac{1}{2} \| S_1 - S_2\|_{q},
$$
where $1/p + 1/q = 1$.
\end{prop}

\begin{proof}
  Let $h \in \mathcal{F}_p$ and define $U_t=\sqrt{t}X_1 + \sqrt{1-t} X_2$ for $t \in [0,1]$. Proceeding as in the proof of \cite[Proposition 5.13]{FHMNP}, we obtain
  \begin{equation*}
    \mathbb{E} \left[ h\left(X_1\right) \right] - \mathbb{E} \left[ h\left(X_2\right) \right]
    =
    \frac{1}{2}
    \int_0^1 \mathbb{E} \left[ \left\langle D^2 h(U_t),S_{1} - S_2 \right\rangle_{HS} \right] dt
\end{equation*}
so that, applying Hölder's inequality for the Schatten norms,
\begin{align*}
  \left|     \mathbb{E} \left[ h\left(X_1\right) \right] - \mathbb{E} \left[ h\left(X_2\right) \right] \right|
  &\leq
  \frac{1}{2} \int_0^1 \mathbb{E} \left[ \operatorname{Tr} \left| \left( D^2 h(U_t)  \right)^{\ast} \left( S_1-S_2 \right)\right|  \right] dt
  \\ &=
  \frac{1}{2} \int_0^1 \mathbb{E} \left[ \left\lVert \left( D^2 h(U_t)  \right)^{\ast} \left( S_1-S_2 \right)  \right\rVert_1  \right]  dt
  \\ &\leq
  \frac{1}{2} \int_0^1 \mathbb{E} \left[ \left\lVert D^2 h(U_t)  \right\rVert_p \left\lVert S_1-S_2 \right\rVert_q \right]  dt
  \\ &\leq
  \frac{1}{2} \left( \sup_{x \in K} \left\lVert D^2 h(x) \right\rVert_p \right) \left\lVert S_1-S_2 \right\rVert_q
  \\ &\leq
  \frac{1}{2} \left\lVert S_1-S_2 \right\rVert_q.
\end{align*}
It remains to take the supremum over all $h \in \mathcal{F}_p$.
\end{proof}

\begin{remark}{\rm For $p=\infty$, Proposition \ref{p:ibp} yields the estimate
\begin{equation}\label{e:glgl}
    \rho_\infty(X_1,X_2) \leq \frac12 \|S_1-S_2\|_1.
\end{equation}
The bound \eqref{e:glgl} is consistent with the fact that $\rho_\infty$ metrizes convergence in distribution on $K$, because of Proposition \ref{p:boga} and the classical inequality
$$
\|\sqrt{S_1} - \sqrt{S_2}\|_{HS}\leq \|S_1-S_2\|_1,
$$
see e.g. \cite[Lemma 4.1]{PS}.}
\end{remark}

\medskip

\noindent{\it End of the proof of Theorem \ref{t:t}}. We reason by contradiction, fix $p \in [1,\infty)$ and assume that $\rho_p$ metrizes indeed convergence in distribution. Let $\{e_i : i\geq 1\}$ be an orthonormal basis of $K$, consider a sequence $\{g_i : i\geq 1\}$ of i.i.d. standard Gaussian random variables and define
$$
X_n := \sum_{i = 1}^{n} \frac{g_i}{n^{1/2}} e_i.
$$
Then, $\{X_n\}$ is a sequence of $K$-valued centered Gaussian elements with covariance kernels $\{S_n\}$ given by $S_n = \sum_{i=1}^{n} \frac{1}{n} \left\langle e_i, \cdot \right\rangle_K e_i$. Therefore, denoting $q= \frac{p}{p-1} > 1$ (with the convention $q=\infty$ when $p=1$), one has

\begin{eqnarray*}
 &&\| S_n\|^q_{q} = \sum_{i=1}^n 1/n^{q} = n^{1-q} \to 0 \quad \mbox{(case $p>1$),}\\
 && \| S_n\|_\infty = \frac1n \to 0 \quad\mbox{(case $p=1$),}
\end{eqnarray*}

and also
\begin{equation*}
  \left\lVert \sqrt{S_n} \right\rVert_{HS}^{2} = \sum_{i=1}^n \frac{1}{n} = 1.
\end{equation*}
By virtue of Proposition \ref{p:boga}, it follows that $X_n$ does not converge in distribution to $Z$, although Proposition \ref{p:ibp} implies that ${\rho}_p(X_n, Z) \to 0$. The conclusion follows. \qed

\smallskip

\begin{remark}{\rm  Note that the argument rehearsed in the above proof does not work for the distance $\rho_{\infty}$, i.e. when $p=\infty$, as in this case $q=1$, so that $\left\lVert S_n \right\rVert_1 = 1 \nrightarrow 0$.}
    \end{remark}

\section{Proof of Lemma \ref{l:test}}\label{s:proofs}

Using the standard chain rule (as stated e.g. in \cite[p. 337]{lang}), one has that
$$
Df(x)(h) = 2 \psi'(\|x - y\|_K^2)\, \langle x-y ,h \rangle_K, x\in K,
$$
so that $f\in C^1(K, \R)$. Applying the product rule (see \cite[p. 336]{lang}) yields also that, for all $x, h, k\in K$,
$$
D^2f(x) (h,k) = 2 \psi'(\|x\|^2) \langle h,k\rangle_K + 4 \psi''(\|x\|^2)\, \langle x-y ,h \rangle_K\, \langle x-y ,k \rangle_K .$$ 
The previous relation shows that $f\in C^2(K, \R)$ and also that, for all $x\in K$, the bilinear form $\varphi_x(h,k) := D^2f(x) (h,k) $ is such that the self-adjoint operator $A_{\varphi_x}$ defined in \eqref{e:aphi} equals (writing once again $I$ for the identity on $K$)
$$
2 \psi'(\|x\|^2) \, I + 4 \psi''(\|x\|^2)\, (x-y)\otimes (x-y),
$$
where $(x-y)\otimes (x-y)$ is shorthand for the rank-one operator $Th = (x-y) \, \langle x-y, h\rangle_K$. The conclusion is obtained by observing that, since ${\rm dim} \, K = +\infty$, the identity is not compact, from which one infers that (because $(x-y)\otimes (x-y)$ is finite rank) $A_{\varphi_x}$ is compact for all $x\in K$ if and only if $\psi'(r) = 0$ for all $r\geq 0$. This concludes the proof.

\qed

\medskip

\section{Acknowledgments}

FB is partially supported by the Italian MUR - PRIN project  no. 2022CLTYP4.\\
SB is supported by the Simons Foundation (Grant 635136).\\
GP is supported by the Luxembourg National Research Fund (Grant:
021/16236290/HDSA).

\printbibliography

\end{document}